\documentclass[a4paper,11pt,twoside]{article}

\usepackage{pstricks-add}
\usepackage{pst-grad}
\usepackage{pst-plot} 
\usepackage{pst-eucl} 
\usepackage{euscript}
\usepackage{epic,eepic}
\usepackage{graphicx}
\usepackage{epsfig}

\usepackage{pst-node}
\usepackage{graphicx}
\newpsstyle{Cblue}{fillstyle=solid,fillcolor=blue!30}
\newpsstyle{Cred}{fillstyle=solid,fillcolor=red!30,shadow=true}

\usepackage{color}
\usepackage{epsfig}                    % usa imagenes PS, EPS
\usepackage{textcomp}
\usepackage{graphicx}
\usepackage{hyperref}
\usepackage[latin1]{inputenc}
\usepackage{amssymb,amsmath,latexsym,amsthm,amsfonts}
\usepackage{pgf}
\usepackage{tikz}

\theoremstyle{plain}
 \newtheorem{theorem}{Theorem}[section]
 \newtheorem{proposition}{Proposition}[section]
 \newtheorem{lemma}{Lemma}[section]
 \newtheorem{corollary}{Corollary}[section]
 
 \newtheorem{definition}{Definition}[section]
  
 \numberwithin{equation}{section}

\newcommand{\be}{\begin{eqnarray}}
\newcommand{\beq}{\begin{eqnarray*}}
\newcommand{\ee}{\end{eqnarray}}
\newcommand{\eeq}{\end{eqnarray*}}
\newcommand{\bp}{\begin{prob}}
\newcommand{\ep}{\end{prob}}
\newcommand{\bs}{\begin{sol}}
\newcommand{\es}{\end{sol}}

\newcommand{\tri}{\triangleleft}
\newcommand{\trieq}{\trianglelefteq}

\newcommand{\email}[1]{\href{mailto:#1}{\nolinkurl{#1}}}

\def\RR{{\mathbb R}}

\def\L{{\mathcal L}}
\def\P{{\mathcal P}}
\def\EE{{\mathbb E}}
\def\PP{{\mathbb P}}

\def\F{{\mathcal F}}

\def\X{{\mathcal X}}

\def\K{{\mathcal K}}

\def\D{{\mathcal D}}

\def\esssup{\text{ess sup\,}} 
\def\essinf{\text{ess inf\,}}

\usepackage{bbm}

\begin{document}

\title{Additive consistency of risk measures and its application to risk-averse routing in networks\thanks{This work was supported by FONDECYT 1100046 and N\'ucleo Milenio Informaci\'on y Coordinaci\'on en Redes ICM/FIC P10-024F.}}
\pagestyle{myheadings}
\markboth{R Cominetti and A. Torrico}{Risk-averse routing in networks}

\author{Roberto Cominetti$^1$
 and Alfredo Torrico$^{2}$
\\[5mm]
\small $\!^1$Universidad de Chile\\
\small Departamento de Ingenier\'{\i}a Industrial\\
\small Santiago, Chile (\email{rccc@dii.uchile.cl})\\[4mm]
\small $\!^2$Universidad de Chile\\
\small Departamento de Ingenier\'{\i}a Matem\'atica\\
\small Santiago, Chile (\email{atorrico@dim.uchile.cl})
}
%\date{\ttfamily Dec 8th, 2013 -- version 1.0}
\date{}
\maketitle

\begin{abstract}
This paper investigates the use of risk measures and
theories of choice for modeling risk-averse route choice 
and traffic network equilibrium with random 
travel times. We interpret the postulates of these theories in the context 
of routing, and we identify {\em additive consistency} as 
a plausible and relevant condition that allows to reduce 
risk-averse route choice to a standard shortest path problem. 
Within the classical theories of choice under risk, we show that 
the only preferences that satisfy this consistency property are the 
ones induced by the entropic risk measures. 
\end{abstract}

%\subjclass[2000]{Primary }
%    For articles to be published after 1 January 2010, you may use
%    the following version:
%\subjclass[2010]{Primary }

%\keywords{Expected utility theory, theories of choice, risk measures, additivity, routing problem, entropic risk measure.}

\section{Introduction}
Drivers are aware that travel time cannot be reliably predicted
and is subject to random fluctuations arising from a multitude of factors such as congestion, 
weather conditions, accidents and traffic incidents, bottlenecks, traffic light disruptions, unexpected actions by
pedestrians and other drivers, and so on. Even on a specific road
segment at a specific time of the day, travel time exhibits a stochastic 
pattern that can be roughly approximated by the log-normal or Burr distributions \cite{denmark, susilawati}.
Thus, choosing a route to travel from a given origin to a destination is essentially a matter of 
comparing random variables. A basic question here is to understand the mechanisms 
by which these choices are made. While this calls for modeling the actual behavior of drivers, it 
can also be approached from a normative angle by asking which are the properties that characterize a
rational route choice under risk.
A related issue is to understand the consequences of risk-averse behavior upon congestion and 
the traffic equilibrium that is obtained.  Answering these questions may change the way in which we 
model traffic and can be relevant for network design and traffic control.

Route preferences vary among individuals and also depending on trip purpose. 
Compare for instance a situation in which you must arrive on time to an important meeting, 
with that of a tourist strolling leisurely through the city, or still a fire truck heading towards an
emergency. While a risk neutral driver may only care about the expected 
travel time, a risk-averse user will be more concerned with travel time reliability. 
Modeling such variety of behaviors has been approached with different tools. Mean-risk models ---with risk 
quantified by the expected value plus the standard deviation--- 
were considered by Nikolova and Stier-Moses \cite{stier1} to study both atomic and non-atomic equilibria.
An algorithm to compute mean-stdev optimal paths was given by Nikolova {\em et al.} \cite{nikolova1, nikolova2}. Route choice using
$\alpha$-percentiles was investigated by Ordo\~nez and Stier-Moses \cite{ordonez1} 
and Nie \cite{nie3}, the former considering also an approach using robust optimization.
Yet another proposal by Nie and Wu \cite{nie1} uses preferences based 
on the on-time arrival probability. An algorithm for this objective function was also given
by Nikolova {\em et al.} \cite{nikolova2}. Finally, Nie {\em et al.} \cite{nie2,wu1} develop a model that uses
stochastic dominance constraints. 
For a more detailed account of these and other relevant references we refer to \S\ref{relatedwork} and to the
literature review included in \cite{stier1}.
%A more detailed description of these works is deferred to \S\ref{relatedwork}.

\begin{figure}[!htb]
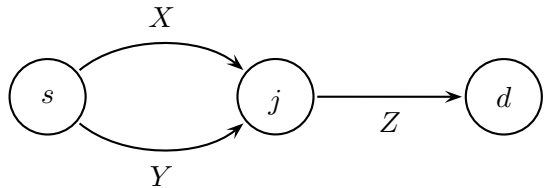

\centering
\psframebox*{%
\rule[-1.2cm]{0pt}{2.4cm}
$\psmatrix[mnode=Circle,radius=5mm,colsep=2.0cm,rowsep=1.5cm,arrowscale=1.5]
[name=1] s &  [name=2] j &  [name=3] d
\ncline[nodesep=1pt]{->}{2}{3}_{Z}
\ncarc[nodesep=1pt,arcangle=-40]{->}{1}{2}_{Y}
\ncarc[nodesep=1pt,arcangle=40]{->}{1}{2}^{X}
\endpsmatrix$ }
\caption{A paradoxical route choice}
\label{fig:fig0}
\end{figure}
In all the approaches just mentioned it may happen that a risk $X$
is preferred to $Y$ but the preference is reversed when we add 
an independent risk $Z$. In the simple network illustrated in figure \ref{fig:fig0}
this means that if we go from $s$ to $j$ our best choice is the upper link, 
but if we extend our trip to $d$ then we must change our choice to the lower link. This
may appear as paradoxical. For a concrete example, consider the mean-stdev 
map $\rho_\gamma^{std}(X)=\EE(X)+\gamma\sigma(X)$ with $\gamma=1$ and 
independent normal variables $X\sim N(11,1)$, $Y\sim N(10,5)$, $Z\sim N(10,2)$,
where $\rho_\gamma^{std}(X)=12<\rho_\gamma^{std}(Y)=10+\sqrt{5}$, but 
$\rho_\gamma^{std}(X+Z)=21+\sqrt{3} >\rho_\gamma^{std}(Y+Z)=20+\sqrt{7}$. 
In this paper we use the theories of choice and risk measures to characterize the 
so-called {\em additive consistent} preferences that are free from these paradoxes. 
We prove that these are exactly the  preferences associated with the {\em entropic risk measures}.

The general theory of {\em choice under risk} is a well established field with a long
history. In this setting, an agent is 
described by a preference relation over a set of random variables 
(or their distributions). Under suitable conditions these preferences can be 
represented by a scalar function. Representations by
expected utilities were already considered by Bernoulli \cite{bernoulli} and further 
developed by Kolmogorov \cite{kolmogorov}, Nagumo \cite{nagumo}, de Finetti \cite{definetti}, 
and von Neumann and Morgenstern \cite{vnm1} (see also \cite[Fishburn]{fishburn2}).
Expected utilities were used by Arrow \cite{arrow1} and Pratt \cite{pratt1} 
to define a local index of {\em absolute risk aversion} that reflects the risk attitudes 
of an agent. However, empirical evidence shows
that agents do not always conform to the postulates of expected utility theory, and
the crucial {\em independence axiom} is sometimes violated (see \cite[Allais]{allais1}, \cite[Ellsberg]{ellsberg1}, 
 \cite[Kahneman and Tversky]{tversky1}). By modifying the independence axiom, several 
alternative representations have emerged: the dual theory of choice by Yaari \cite{yaari1}, 
the anticipated utility theory by Quiggin \cite{quiggin1}, the rank-dependent expected utility theory 
by  Wakker \cite{wakker1} and Chateauneuf \cite{chate1}, and Schmeidler's approach \cite{schmeidler1} 
based on subjective probabilities.

On the other hand, the extensive use of Value-at-Risk in finance gave birth to the notion of {\em risk measure} as an
alternative tool for studying choice under risk. The axiomatic approach to risk measures was initiated by 
Artzner {\em et al.} \cite{artzner1}, who also introduced  the Average Value-at-Risk 
as a {\em coherent} risk measure that overcomes some limitations of Value-at-Risk.
Mean-risk functionals have also been considered in this context, notably by Ogryczak and 
Ruszczynski \cite{rusz4} who studied a risk measure that combines the expected value
and the standard semi-deviations. To some extent, risk measures can be unified with the theories of choice  through 
the concept of premium principles (see Gerber \cite{gerber1}, Goovaerts {\em et al.} 
\cite{goovaerts2,goovaerts4}, Denuit {\em et al.} \cite{denuit1}, and Tsanakas and Desli \cite{tsanaka1}).
%Also relevant in this context is the {\em index of riskiness} recently introduced by Aumann and Serrano \cite{aumann}.
A recent account of theories of choice and risk measures can be found in the book by F\"ollmer and Schied
 \cite{follmer3}.

 {\sc Our contribution:} 
In this paper we investigate the use of risk  measures and theories of choice to model risk-averse routing. 
The interpretation of the postulates in this context leads us to identify {\em additive consistency} as a plausible 
and relevant condition that extends the notion of {\em translation invariance} and reduces risk-averse route 
choice to a standard shortest  path problem. We briefly discuss how this allows to formulate risk-averse 
equilibrium models for atomic and non-atomic network flows, which naturally fit in the framework of congestion 
games. We then investigate additive consistency in some standard settings of theories 
of choice proving that, within the classes of {\em distorted risk measures} as well as 
{\em rank dependent utilities}, the only maps that satisfy additive consistency  are the entropic risk measures. We also show that these
are the only {\em expected utility} maps that are translation invariant, hence the only 
risk measures in this class. These results extend Gerber \cite{gerber1}, Goovaerts {\em et al.} \cite{goovaerts1}, Heilpern \cite{heil1}, and Luan \cite{luan1}.

 {\sc Structure of the paper:} In section \S\ref{riskmeasures} we recall the postulates of 
risk measures and their induced preferences, interpreting them in the context of 
route choice. We introduce the concept of {\em additive consistency} and discuss 
its application to risk-averse path choice and network equilibrium. In \S\ref{Stres}
we consider consecutively the classes of expected utility maps (\S\ref{exputilities}), 
distorted risk measures (\S\ref{dualchoice}), and rank dependent utilities (\S\ref{rankutilities}), proving that 
within each of these classes the entropic risk measures are the only ones that satisfy translation 
invariance and/or additive consistency. In \S\ref{dynrisk} we make some remarks on the use 
of dynamic risk measures as an alternative to model route choice, and we conclude in \S\ref{relatedwork} 
with a brief discussion of related work.

\section{Risk measures and additive consistency}\label{riskmeasures}
Quantifying risk is an essential yet difficult task. Because of the subjective nature of risk perception,
defining an appropriate measure remains controversial and several approaches have been proposed
each one with its own advantages and limitations. A risk quantification   
attaches a scalar value to each random variable $X:\Omega\to\RR$, where $\Omega$ is a set of 
events endowed with a $\sigma$-algebra $\F$  and a probability measure $\PP$. 
%At this level $X$ could represent a random travel time or the uncertain value of a portfolio. 
More precisely, a {\em risk measure} is a map $\rho:\X\to\RR$ defined over a {\em prospect space} $\X$ (a linear space of random variables containing the constants, usually a subspace of $L^\infty(\Omega,\F,\PP)$) which satisfies the following postulates:\vspace{-2ex}
\begin{itemize}
\item {\sc normalization:} $\rho(0)=0$,\vspace{-1ex}
\item {\sc monotonicity:}  if $X\leq Y$ almost surely then $\rho(X)\leq\rho(Y)$,\vspace{-1ex}
\item {\sc translation invariance:}  $\rho(X+m)=\rho(X)+m$ for all $m\in\RR$.\vspace{-1ex}
\end{itemize}
Such a map induces a preference relation $X\trieq Y\Leftrightarrow\rho(X)\leq\rho(Y)$ which 
defines a complete order. In this paper prospects are interpreted as costs or disutilities so that 
smaller values are preferred and, against common usage, $X\trieq Y$ is read 
as {\em ``$X$ is preferred to $Y$''}. Naturally, the normal convention applies if $X$ represents 
a utility and larger values are better. We use $X\tri Y$ to denote strict preference and we write $X\sim Y$ 
when simultaneously $X\trieq Y$ and $Y\trieq X$.

The normalization axiom is not restrictive as one can always take $\rho(X)-\rho(0)$ instead of 
$\rho(X)$. Monotonicity has a clear intuitive meaning:  
larger costs convey higher risk. In the context of routing, paths with larger travel times are 
riskier and less preferred. Translation invariance 
is equivalent (under normalization) to requiring simultaneously \vspace{-2ex}
\begin{itemize}
\item {\sc normalization on constants:} $\rho(m)=m$ for all $m\in\RR$.\vspace{-1ex}
\item {\sc translation consistency:}  $\rho(X)\leq\rho(Y)\Rightarrow\rho(X+m)\leq\rho(Y+m)$.\vspace{-2ex}
\end{itemize}
The latter is a plausible  condition stating that preferences between prospects are not 
altered when we add them a constant.  While this postulate is not universally accepted in finance 
(attitudes towards risk might change after receiving a heritage), it seems very likely in the context of route choice
(see \S\ref{routechoice}).
Finally, normalization on constants is also a mild requirement: it suffices to have 
$m\mapsto \rho(m)$ strictly increasing and continuous, since then this function 
has an inverse $\sigma$ and we may substitute $\rho$ by $\sigma\circ\rho$.

The axiomatic approach to risk measures was initiated by \cite[Artzner {\em et al.}]{artzner1} 
who introduced the notion of a {\em coherent risk measure}, namely, a risk 
measure which is also {\em sub-additive} and {\em positively homogeneous}. 
%Such measures are closely connected with their acceptance set $\A=\{X\in\X:\rho(X)\leq 0\}$
%from which they can be recovered as $\rho(X)=\inf\left\{m\in\RR: \ X+m\in\A\right\}$. 
Positive homogeneity translates the notion of scale invariance, while sub-additivity 
captures the idea that a merger of two risks cannot create additional risk. The validity 
of these axioms in finance has been thoroughly debated in the literature.
In the context of route choice these assumptions seem less natural, specially positive 
homogeneity.  A weaker property is {\em convexity} which still supports a useful dual representation 
for risk measures \cite{follmer1,follmer2,follmer3}.
%\begin{itemize}
%	\item {\bf Positive Homogeneity}: \textit{``The magnitude of the random variable affects proportionally to the risk''. For all $ \lambda\geq0$, $X\in\X$, $\rho(\lambda X)=\lambda\rho(X)$. }
%	\item {\bf Subadditivity}: \textit{``A merger does not create extra risk''. $\forall X,Y\in\X$ we have $\rho(X+Y)\leq\rho(X)+\rho(Y)$.}
%\end{itemize}
%With these four axioms, Artzner et al. in \cite{artzner1} define a class of risk measures called \textit{coherent risk measures}.
%\begin{definition}[Coherent risk measure]
%A risk measure satisfying positive homogeneity and subadditivity is called coherent.
%\end{definition}
%One of the most well-known and used coherent risk measure is the Average Value-at-Risk, which is explained in the following section. 

Preferences can also be modeled directly as a preorder, namely
 a reflexive and transitive relation $\trieq$. 
We then say that $\trieq$ is represented by $\rho:\X\to\RR$ if $X\trieq Y\Leftrightarrow \rho(X)\leq \rho(Y)$. 
The following simple result clarifies when both modeling approaches coincide.  
\begin{proposition}\label{propo2}
Suppose $\trieq$ is a preorder that satisfies
\\[1ex]
$\bullet$ {\sc monotonicity:} if $X\leq Y$ almost surely then $X\trieq Y$,
\\[1ex]
$\bullet$ {\sc translation consistency:} if $X\trieq Y$ then $X+m\trieq Y+m$ for $m\in\RR$,
\\[1ex]
$\bullet$ {\sc real ordering:} for $X,Y$ constant we have $X\trieq Y\Leftrightarrow X\leq Y$,
\\[1ex]
$\bullet$ {\sc scalarization:} for each $X\in\X$ there is a unique $\alpha\in\RR$ 
with $X\sim \alpha$.
\\[1ex]
Then the map $X\mapsto\rho(X)=\alpha$ defined by the last condition is a risk measure and
$\trieq$ is represented by $\rho$.
\end{proposition}
\proof  Reflexivity of $\trieq$ gives $X\sim X$ so that for $X\equiv m$ constant
we get $\rho(m)=m$ which shows that $\rho$ is normalized on constants. 
Next, by definition we have $X\sim\rho(X)$ and translation consistency gives 
$X+m\sim \rho(X)+m$ so that $\rho(X+m)=\rho(X)+m$ proving the translation invariance 
of $\rho$. Since the monotonicity of $\trieq$ readily implies the monotonicity of $\rho$, it 
follows that $\rho$ is a risk measure. It remains to establish the representation property.
By definition we have $X\sim\rho(X)$ and $Y\sim\rho(Y)$ so that transitivity
gives $X\trieq Y$ iff  $\rho(X)\trieq\rho(Y)$. According to the real ordering axiom,
the latter is equivalent to $\rho(X)\leq\rho(Y)$.
\endproof
As a corollary to Proposition \ref{propo2} the preorder $\trieq$ must be complete, that is, 
all pairs are comparable. This also follows directly from the real ordering and scalarization 
axioms. Scalarization is a non-trivial condition. In section \S\ref{Stres} we will
revise this postulate in the light of the theories of choice.

\subsection{Examples and counterexamples of risk measures}\label{examples}

A first attempt to quantify risk was given in \cite[Markowitz]{markowitz1} by considering  the 
mean-risk functional
$$\rho_\gamma^{var}(X)=\mu(X)+\gamma\sigma^2(X)$$
with $\mu(X)$ the mean of $X$, $\sigma^2(X)$ its variance, and $\gamma>0$ a positive constant. 
Variations of this idea substitute the variance by the standard deviation
$$\rho_ \gamma ^{std}(X)=\mu(X)+\gamma\sigma(X)$$
or other variability measures such as the absolute semi-deviations
\cite{rusz4,ogr1}. While these maps satisfy normalization and translation invariance, they are 
not risk measures since monotonicity might fail: take  $X\sim U[0,1]$ a uniform variable and 
$Y= (1+X)/2$ so that $X\leq Y$ almost surely, yet for $\gamma$ large we have 
$\rho_ \gamma ^{var}(Y)<\rho_ \gamma ^{var}(X)$ and the same for $\rho_ \gamma ^{std}$.

A popular measure is \textit{Value-at-Risk} defined for $p\in(0,1)$ as the percentile
$$\text{VaR}_p(X)=\inf\left\{m\in\RR: \PP(X\leq m)\geq 1-p\right\}.$$
This is a risk measure which is also positively homogeneous, but not convex  nor sub-additive (see \cite{artzner1}).
The best known coherent risk measure is \textit{Average Value-at-Risk},
introduced in \cite{artzner1} and defined for a level $p\in(0,1)$ by
$$\text{AVaR}_p(X)=\frac{1}{p}\int^p_0\text{VaR}_q(X)dq,$$
which also has the following useful dual representation ({\em cf.} \cite{follmer1, follmer2,rocka1,rocka2}) 
$$\text{AVaR}_p(X)=\mbox{$\frac{1}{p}$}\inf_{z\in\RR}\left\{\EE((X-z)_{+})+pz\right\}.$$ 
AVaR is also known by the names of \textit{Conditional Value-at-Risk}, \textit{Tail Value-at-Risk}, and \textit{Expected Shortfall}.
For continuous variables it coincides with the \textit{Tail Conditional Expectation}
$$\text{TCE}_p(X)=\EE(X|X\geq \text{VaR}_p(X)).$$ 
Note that when restricted to normal random variables both 
VaR$_p$ and AVaR$_p$ coincide with $\rho_\gamma ^{std}$
for appropriate corresponding constants $\gamma$. 

A family of convex (but not coherent) risk measures are the entropic measures defined as
({\em cf.} \cite{follmer1,follmer2,follmer3,rusz1})
$$\rho^{ent}_{\beta}(X)=\mbox{$\frac{1}{\beta}$}\ln(\EE(e^{\beta X})).$$
These measures play a central role in our results. They can be 
derived from additive premium principles \cite[Gerber]{gerber1}, as well as from expected 
utilities with constant absolute risk aversion CARA (\cite[Arrow]{arrow1}, \cite[Pratt]{pratt1}). 
The case $\beta>0$ characterizes risk-averse behavior while $\beta<0$ 
corresponds to a risk-prone agent. The limit $\beta\to 0$ gives
$\rho^{ent}_0(X)=\EE(X)$ which reflects risk neutrality, while 
$\beta\to\pm\infty$ yields extreme attitudes toward risk with 
$\rho^{ent}_{\infty}(X)=\esssup X$ and $\rho^{ent}_{-\infty}(X)=\essinf X$. 
In the sequel we only consider finite $\beta$'s and exclude the last two.

\subsection{Risk measures and consistency in route choice}\label{routechoice}
Consider a driver who must choose one among a finite set of routes,
each of which has a random travel time in a suitable prospect space $\X$.  
While a risk-neutral driver may prefer the route with smallest expected time 
(easily computed by any shortest path algorithm), a risk-averse user might be
willing to trade some expected value against increased reliability. 

We assume that the driver preferences $\trieq$ satisfy the axioms in Proposition \ref{propo2}, 
so that route choice is based on a risk measure
$\rho:X\to\RR$. As already mentioned, the scalarization 
postulate is a nontrivial assumption which will be discussed later. 
In contrast, the axioms of monotonicity and real ordering seem 
quite innocuous, while translation consistency is also very plausible in this context. 
Namely, consider the simple network illustrated in figure \ref{fig:fig1}
\begin{figure}[!htb]
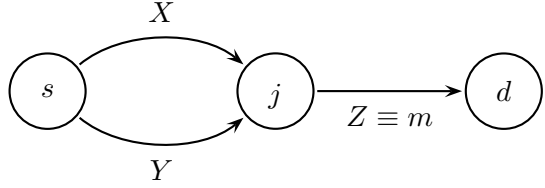

\centering
\psframebox*{%
\rule[-1.2cm]{0pt}{2.4cm}
$\psmatrix[mnode=Circle,radius=5mm,colsep=2.0cm,rowsep=1.5cm,arrowscale=1.5]
[name=1] s &  [name=2] j &  [name=3] d
\ncline[nodesep=1pt]{->}{2}{3}_{Z\equiv m}
\ncarc[nodesep=1pt,arcangle=-40]{->}{1}{2}_{Y}
\ncarc[nodesep=1pt,arcangle=40]{->}{1}{2}^{X}
\endpsmatrix$ }
\caption{Translation invariance and additive consistency}
\label{fig:fig1}
\end{figure}
with two paths from $s$ to $j$ with random  times $X$ and $Y$, followed by a single path from 
$j$ to $d$ with constant time $Z\equiv m$. 
Translation consistency simply requires that a driver  who prefers $X$ to $Y$ for moving from $s$ to 
$j$, should have the same preference when heading towards $d$. 

A stronger consistency property requires the preservation of preferences 
when $Z$ is no longer constant but still
independent from $X$ and $Y$, namely, if $X\trieq Y$ then $X+Z\trieq Y+Z$ for all $Z\perp (X,Y)$.
Intuitively, since the arc $(j,d)$ is compulsory and one must inevitably pass through it, the
decision at $s$ should not depend on $Z$. This seems all the more plausible since, due to the independence,
even if one observes $Z$ this reveals no information that 
could affect the choice between $X$ and $Y$. 
This motivates the following definition.
\begin{definition}[Additive consistency]
A map $\rho:\X\to\RR$ is called {\em additive consistent} if for all $X,Y,Z\in\X$ with $Z\perp (X,Y)$
we have
$$\rho(X)\leq\rho(Y) \ \Rightarrow \ \rho(X+Z)\leq\rho(Y+Z).$$ 
\end{definition}
For risk measures this is equivalent to an apparently stronger requirement of additivity
for sums of independent risks. The proof is elementary.
\begin{lemma}\label{Lema1}
Let $\rho:\X\to\RR$ be a risk measure. Then $\rho$ is additive consistent if and only if 
$\rho(X+Y)=\rho(X)+\rho(Y)$ for all $X,Y\in\X$ with $X\perp Y$. A map satisfying the latter 
is called {\em additive}.
\end{lemma}
\proof
The {\em ``if''} part is obvious so we just prove the {\em ``only if''}\!. 
Let $X\perp Y$. Since $X\sim\rho(X)$, from additive consistency we get $X+Y\sim \rho(X)+Y$.
Hence $\rho(X+Y)=\rho(\rho(X)+Y)$ and the translation invariance of $\rho$ yields
$\rho(X+Y)=\rho(X)+\rho(Y)$.
\endproof

It is well known that the entropic risk measures $\rho^{ent}_\beta$ are additive and hence 
additive consistent. The counterexample in the Introduction (see figure \ref{fig:fig0}) 
shows that this is not the case for $\rho^{std}_\gamma$. The example was for $\gamma=1$ but
it can be readily adapted to any $\gamma>0$. Thus, in general $\rho^{std}_\gamma$ is not 
additive consistent, and {\em a fortiori} neither VaR$_p$ nor AVaR$_p$ since they coincide
 with $\rho^{std}_\gamma$ for normal variables. In section \S\ref{Stres} we show that, among a wide class of risk 
measures, the entropic ones are the only that are additively consistent.

\noindent{\sc Remark.} The use of normal distributions in the counterexample in the Introduction 
could raise some objections since these are unbounded and have positive mass on the 
negative reals, so they might not represent travel times. However, the example is robust and can be modified to get distributions with 
bounded support on $\RR_+$:  it suffices to shift the variables by a large common 
constant so that the mass on $\RR_-$ becomes negligible, and then truncate to a large 
interval $[0,M]$ and take conditional distributions.

\subsection{Application to risk-averse network equilibrium} 
Additive consistency is a plausible assumption with interesting consequences 
for the computation of risk-minimizing routes and risk-averse network equilibrium.
Consider a network $G=(V,A)$ in which every link $a\in A$ has a random
travel time $\tau_a$ and assume that these variables are independent.
Let $\P$ be the set of paths connecting a given origin $s$ to 
a destination $d$, and for each $p\in\P$ denote $T_p=\sum_{a\in p}\tau_a$ 
the corresponding travel time. Given a risk measure $\rho$ we consider the 
problem of finding a risk-minimizing path
\be\label{PPP}
\min_{p\in\P}~\rho(T_p).
\ee
When $\rho$ is additive the objective function separates
as $\rho(T_p)=\sum_{a\in p}\rho(\tau_a)$ and (\ref{PPP}) reduces to 
a standard shortest path problem with arc lengths $w_a=\rho(\tau_a)$.
This can be efficiently solved using standard algorithms.

Consider now a non-atomic equilibrium problem with 
traffic demands $g_k\geq 0$ for a family of origin-destination 
pairs $(s_k,d_k)_{k\in\K}$. The demands decompose into
path-flows $x_p\geq 0$ so that $g_k=\sum_{p\in\P_k}x_p$ 
where $\P_k$ denotes the set of paths connecting $s_k$ to $d_k$.
The cumulative flow on a link $a\in A$ is then $y_a=\sum_{p\ni a}x_p$
where the sum extends to all paths $p\in\cup_{k\in\K}\P_k$
containing $a$. Suppose that the distribution $\tau_a\sim\F_a(y_a)$ 
depends on the total link flow $y_a$. We may then define a {\em risk-averse 
network equilibrium} as a path-flow vector $x$ which uses only risk-minimizing 
paths, namely, for each OD pair $k\in\K$ and every path $p\in\P_k$ we must have 
$$x_p>0\Rightarrow \rho(T_p)=\min_{r\in\P_k}\rho(T_r).$$
If $\rho$ is additive and the function $\sigma_a(y_a)\triangleq \rho(\tau_a)$ 
increases with $y_a$, this reduces to a standard Wardrop equilibrium 
and equilibria are characterized as the optimal solutions of the convex program
\be\label{riskequilibrium}
\min_{(x,y)\in F}\sum_{a\in A}\int_0^{y_a}\!\!\!\sigma_a(z)\,dz
\ee
where $F$ stands for the set of all feasible flows satisfying flow conservation.

A similar model can be stated in the atomic case with finitely many players.
Each player $i\in I$ choses a path $p_i$ from his origin to his destination
and gets  $\rho(T_{p_i})$ as payoff. Assuming that 
the distribution of $\tau_a\sim \F_a(n_a)$ depends on the number of players that use 
the link $n_a=|\{i\in I:a\in p_i\}|$, and denoting $\sigma_a(n_a)=\rho(\tau_a)$, this 
yields a congestion game which falls in the framework of Rosenthal and admits the potential function
\be\label{pot}
\Phi(p_i:i\in I)=\sum_{a\in A}\sum_{z=0}^{n_a}\sigma_a(z).
\ee
In both the atomic and non-atomic settings above, all drivers were assumed homogeneous with 
respect to their valuation of risk. In the next section we show that additive consistency 
limits the choice to entropic risk measures so that some similarity among users 
might be expected, nevertheless they can still differ in their absolute risk aversion index.
The latter calls for an equilibrium model 
with multiple user classes, for which one can still establish the existence of equilibria 
but a simple variational characterization such as (\ref{riskequilibrium}) or the existence of a potential function 
like (\ref{pot}) seems unlikely.

%\textcolor{red}{>Se puede decir algo mas para el caso de aversi\'on al riesgo heterog\'eneas?}

\section{Theories of choice and additive consistency}\label{Stres}

The scalarization postulate in Proposition \ref{propo2}
is a nontrivial assumption that needs further justification. The theories of choice 
provide sufficient conditions for this property to be satisfied.
In particular, a preorder $\trieq$ on a topological 
space $\X$ has a scalar representation $C:\X\to\RR$ if and only if  
there is a countable dense subset $\D\subset \X$ such that whenever 
$X\tri Y$ one can find $Z\in\D$ with $X\trieq Z\trieq Y$
(see \cite[Theorem 2.6]{follmer3} and references therein).
Unfortunately this is not enough for our purposes
and additional conditions are needed to get equivalence to a constant $X\sim \alpha$. 
This can be achieved when $\X$ is a prospect space of random variables, in which case 
more specific formulas for $C(X)$ can be obtained.

Already  in the 18th century, Daniel Bernoulli \cite{bernoulli} observed that preferences on
prospects could be represented by an expected utility $C(X)=\EE(c(X))$.
Axiomatic approaches for this type of representation were developed among others by 
Kolmogorov \cite{kolmogorov}, Nagumo \cite{nagumo}, de Finetti \cite{definetti}, 
and von Neumann \& Morgenstern \cite{vnm1}. Here we consider a version  given 
by F\"ollmer \& Schied \cite{follmer3}. More recently, alternative representations have 
been obtained under different sets of axioms, including the dual theory of choice
\cite[Yaari]{yaari1} and the rank-dependent utilities \cite[Quiggin]{quiggin1}, 
\cite[Wakker]{wakker1}, \cite[Chateauneuf]{chate1}. 

Throughout this section we consider a preorder  $\trieq$  defined on the whole space \mbox{$\X=L^\infty(\Omega,\F,\PP)$}
where $(\Omega,\F,\PP)$ is a standard atomless probability space.
We denote $\D_b$ the set of probability distributions on $\RR$ with bounded support so that
each $X\in\X$ has a distribution $F_X\in\D_b$ and, by Skorohod's representation theorem
\cite{skorohod}, all distributions in $\D_b$ are obtained in this way.
On $\X$ we consider the convergence for the $L^1$-norm, denoted $X_n\rightarrow X$,
as well as the convergence in distribution: $X_n\stackrel{\D}{\rightarrow} X$ iff  $F_{X_n}$ converges weakly to $F_X$, that 
is $\int_\RR\varphi(x)\,dF_{X_n}(x)\to\int_\RR\varphi(x)\,dF_X(x)$ 
for all bounded continuous functions $\varphi:\RR\to\RR$.
% or equivalently $F_{X_n}(x)\to F_X(x)$ at every continuity point $x$ of $F$. 
%This is equivalent to the weak convergence of the measures.

\subsection{Expected utility}\label{exputilities}

According to \cite[Corollary 2.29]{follmer3}, % (see also \cite[Theorem 3.5]{rusz5}),  
a preorder $\trieq$ over $\X$ admits an expected utility representation of the form\footnote{Since 
in our setting $X$ represents a cost, it might be more appropriate to call it {\em expected disutility} or {\em expected cost}, 
but we adhere to the standard terminology.}
$$C(X)=\EE(c(X))=\int_{\RR}c(x)\,dF_X(x),$$
with $c:\RR\to\RR$ strictly increasing and continuous (unique up to a positive affine transformation)
if and only if the following axioms are satisfied\vspace{-2ex}
\begin{description}
\item[(A$_1$)] {\sc law invariance\footnote{Note that in this case $\trieq$ induces a preorder $\preceq$ 
on $\D_b$ by $F_X \preceq F_Y \Leftrightarrow X\trieq Y.$}:} $F_X=F_Y\Rightarrow X\sim Y$. \vspace{-0.5ex}
\item[(A$_2$)] {\sc weak continuity:} the sets  $\{Y\in\X: Y\trieq X\}$ and $\{Y\in\X: X\trieq Y\}$ are closed
for  convergence in distribution.\vspace{-0.5ex}
\item[(A$_3$)] {\sc independence}: if $X\trieq Y$ then 
 $\L(p;X;Z)\trieq \L(p;Y;Z)$ for all $Z\in\X$ and $p\in[0,1]$.
Here $\L(p;X;Z)$ denotes the lottery with distribution given by
$\alpha F_X(x)+(1-\alpha)F_Z(x)$ for all $x\in\RR$.\vspace{-2ex}
\end{description}
This is a general version of the von Neumann and Morgenstern representation result \cite{vnm1}, originally 
stated for lotteries over a finite event space. For further discussions see \cite{rusz5,fishburn1,fishburn2,follmer3}. 

In the context of route choice, expected utility preferences
hold an intuitive appeal. Imagine for instance a fire truck rushing towards 
an emergency. Clearly enough, reaching the destination as quickly as possible is critical,
all the more since the damage caused by fire increases non-linearly with time. 
A route with small expected time but affected by events of high congestion 
might be too risky, and a longer but more reliable route could be a better choice.
Expected utility captures the nonlinear relation between ``time'' and ``cost'', so that
minimizing expected cost seems a reasonable model for the actual behavior of firemen.

The  properties of the utility $c(\cdot)$ are naturally connected to those of $\trieq$.
For instance,  $c(\cdot)$ is convex iff $\trieq$ is \textit{risk-averse} in the sense 
that the expected value of a prospect is always preferred to the prospect itself: $\EE(X)\trieq X$.
Also $c(\cdot)$ is increasing iff  $\trieq$ is 
monotone, and strictly increasing if  $X\tri Y$ whenever $X< Y$ almost surely. 
In this latter case $c(\cdot)$ has an inverse and one can also represent $\trieq$ by 
the so-called {\em certainty equivalent}
$$\rho_c(X)=c^{-1}(\EE(c(X))).$$
Note that taking $\alpha=\rho_c(X)$ we have $C(\alpha)=c(\alpha)=C(X)$
so that $X\sim\alpha$ and the scalarization postulate in Proposition \ref{propo2} holds true.
Moreove, note that if a preorder $\trieq$ is induced by a map $\rho:\X\to\RR$ and  satisfies (A$_1$)-(A$_3$), 
then $\rho$ is necessarily of the form $\rho_c$. 
For completeness we state this explicitly.
\begin{corollary} Let $\rho:\X\to\RR$ be such that it satisfies 
\\[1ex]
$\bullet$ {\sc law invariance:} $F_X=F_Y\Rightarrow \rho(X)=\rho(Y)$,
\\[1ex]
$\bullet$ {\sc normalization on constants:} $\rho(m)=m$ for all $m\in\RR$,
\\[1ex]
$\bullet$ {\sc strict monotonicity:} if $X<Y$ almost surely then $\rho(X)<\rho(Y)$,
\\[1ex]
$\bullet$ {\sc weak continuity:} if $X_n\stackrel{\D}{\rightarrow} X$ then $\rho(X_n)\to \rho(X)$,
\\[1ex]
$\bullet$ {\sc independence:} if $\rho(X)\leq\rho(Y)$ then $\rho(\L(p;X;Z))\leq\rho(\L(p;Y;Z))$.
\\[2ex]
Then $\rho=\rho_c$ for some $c:\RR\to\RR$ strictly increasing and continuous, and 
unique up to a positive affine transformation.
\end{corollary}
\proof Let us consider the induced order $X\trieq Y\Leftrightarrow \rho(X)\leq\rho(Y)$.
The assumptions imply that $\trieq$ is represented by an expected utility map $\rho_c$. 
Normalization on constants gives $\rho(X)=\rho(\rho(X))$ from which we
deduce $X\sim\rho(X)$ and therefore $\rho_c(X)=\rho_c(\rho(X))=\rho(X)$.
\endproof

In general $\rho_c$ is not a risk measure since translation invariance may fail. 
We show next that this only holds for the entropic risk measures. 
This result goes back to \cite[Gerber]{gerber1} where it was proved under the stronger 
condition of additivity of $\rho_c$ and assuming $c(\cdot)$ concave 
non-decreasing and twice differentiable.
Under translation invariance, but still assuming regularity of $c(\cdot)$, the result was 
proved in \cite[Luan]{luan1} (see also \cite[Heilpern]{heil1}).
Our proof below relies exclusively on continuity and monotonicity. 
Regularity of $c(\cdot)$ as well as additivity of $\rho_c$ are obtained as a consequence.

\begin{theorem}\label{theorem: vnm2}
The only translation invariant maps of the form $\rho_c$ with 
$c:\RR\to\RR$ strictly increasing and continuous, are the entropic risk measures
$$\rho^{ent}_{\beta}(X)=\left\{\begin{matrix} \frac{1}{\beta}\ln(\EE(e^{\beta X}))& \text{if} \ \beta\neq0\\[1ex]
\EE(X) & \text{if} \ \beta=0.\end{matrix}\right.$$
\end{theorem}
\begin{proof}[Proof] Since $\rho_c$ does not change under affine transformations of $c(\cdot)$,  we may
assume $c(0)=0$. The translation invariance of $\rho_c$ gives 
$$c^{-1}\left(\EE(c(X+m))\right)=c^{-1}\left(\EE(c(X))\right)+m.$$
Take $X=zB_p$ with $z\in\RR$ and $B_p$ a Bernoulli variable with parameter $p$. 
Developing the left and right hand sides, and using the fact that $c(0)=0$, this equality becomes
$$c^{-1}\left(pc(z+m)+(1-p)c(m)\right)=c^{-1}\left(pc(z)\right)+m.$$
%\be\label{eq: cost} p[c(z+m)-c(m)]=c(c^{-1}\left(pc(z)\right)+m)-c(m).\ee
Defining $\varepsilon=c^{-1}(pc(z))$ this can be rewritten as
\be\label{eq:cost2} c(\varepsilon)[c(z+m)-c(m)]=c(z)[c(m+\varepsilon)-c(m)].\ee
Now, since $c(\cdot)$ is increasing it is differentiable almost everywhere.
Take any point $m$ at which $c'(m)$ exists. By considering alternately $z>0$ and $z<0$ with $p\to 0^+$
we have respectively $\varepsilon\to 0^+$ and $\varepsilon\to 0^-$.
Dividing (\ref{eq:cost2}) by $\varepsilon$ and noting that $c(z)\neq 0$ and 
$[c(z+m)-c(m)]\neq 0$, it follows that the lateral derivatives of $c(\cdot)$ at 0 exist 
and coincide, and moreover we have
\be\label{eq:cost3} c'(0)[c(z+m)-c(m)]=c(z)c'(m).\ee
Now that we know that $c'(0)$ exists, we may reuse (\ref{eq:cost2}) and apply a similar argument
at an arbitrary point $m$ to deduce that $c'(m)$ exists everywhere and satisfies (\ref{eq:cost3}). 
Moreover, since $c(\cdot)$ is strictly increasing the mean value theorem
implies that $c'(m)>0$ at some point $m$, and (\ref{eq:cost3}) yields $c'(0)>0$.
Using an affine transformation (which does not affect $\rho_c$) we may assume $c'(0)=1$
and then rearranging (\ref{eq:cost3}) we get  
\be \label{eq:cost4} c(z+m)-c(z)=c(m)+[c'(m)-1]c(z).\ee
Dividing by $m>0$ and letting it to 0 it follows that $[c'(m)\!-\!1]/m$ has a limit, which we denote by $\beta$, 
and $c(\cdot)$ satisfies the differential equation $$c'(z)=1+\beta c(z).$$ 
This has a unique continuous solution with $c(0)=0$, namely  $c(x)=x$ if $\beta=0$
and $c(x)=[e^{\beta x}-1]/\beta$ otherwise. The conclusion follows.
\end{proof}

%\textcolor{red}{>Se puede localizar este resultado? >Vale para $\X$ convexo por ejemplo?}

\subsubsection{The independence axiom and Allais' paradox}
Expected utility theory has not been without critics, mainly focusing
on the independence axiom. The paradoxes of Allais \cite{allais1} and Ellsberg 
\cite{ellsberg1} show specific contexts in which the independence axiom is violated
and agents do not behave consistently with the predictions of this theory. Further 
empirical evidence has been provided by Kahneman and Tversky \cite{tversky1}. 
\begin{figure}[!h]
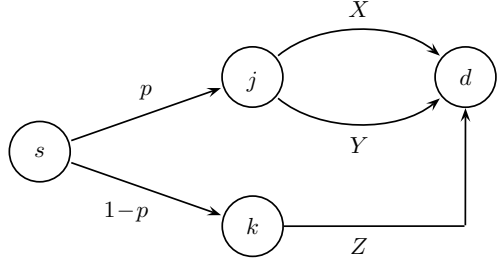

\centering
\scalebox{0.8}{
\psframebox*{%
\rule[0cm]{0pt}{3.5cm}
$\psmatrix[mnode=Circle,radius=5mm,colsep=2.5cm,rowsep=0.2cm,arrowscale=1.5]
  & [name=2] j &  [name=3] d\\
 [name=1] s  \\
& [name=4] k & [mnode=Circle,radius=0mm,name=5] 
\ncline[nodesep=1pt]{->}{1}{2}^{p}
\ncline[nodesep=1pt]{->}{1}{4}_{\hspace{-4ex}1\!-\!p}
\ncarc[nodesep=1pt,arcangle=-40]{->}{2}{3}_{Y}
\ncarc[nodesep=1pt,arcangle=40]{->}{2}{3}^{X}
\ncline[nodesep=0pt]{-}{4}{5}_{Z}
\ncline[nodesep=0pt]{->}{5}{3}
\endpsmatrix$ }
}
\caption{The independence axiom and Allais' paradox.}
\label{fig:fig3}
\end{figure}

To interpret the independence axiom,
imagine a driver who has two options $X,Y$ to travel from $j$ to $d$ of which
he prefers $X$ (see figure \ref{fig:fig3}). Suppose now that he is actually at a point $s$ 
on the other side of a river, and to reach the intermediate node $j$ he must first cross a bridge which 
is open with probability $p$, and else take a long detour $Z$ to the destination $d$. 
Thus, the driver faces a choice between the lotteries $\L(p;X;Z)$ and $\L(p;Y;Z)$. The 
independence axiom postulates that the first should be preferred.
While this seems a reasonable assumption in the route choice setting, Allais observed that 
it may fail in other contexts. Specifically, he considered
$$X=50, \ Y=\left\{\begin{matrix}35& \text{with probability} \ 0.8\\100& \text{with probability} \ 0.2\end{matrix}\right., \ Z=100.$$
and noted that while most people prefer $X$ to $Y$, for $p=0.25$ they tend to choose $\L(p;Y;Z)$ over $\L(p;X;Z)$ where
\begin{align*}
\L(p;X;Z)=&\left\{\begin{matrix}50& \text{with probability} \ 0.25,\\100& \text{with probability} \ 0.75,\end{matrix}\right.\\
\L(p;Y;Z)=&\left\{\begin{matrix}35& \text{with probability} \ 0.2,\\100& \text{with probability} \ 0.8.\end{matrix}\right.
\end{align*}
This points to a potential incongruence between the predictions based on the independence axiom 
and the actual choices made by agents. However, it is also true that {\em ``context matters''} and 
decisions depend not only on the way the choice is formulated but even on the form in which 
information is communicated and processed. The lotteries above do not describe the route 
choice accurately since they obscure the fact that here we face a two-stage decision process with 
the possibility of recourse: the choice between $X$ or $Y$ --- and $Z$ --- can be postponed until 
we know whether the bridge is open. 
Stated in this way the observed inconsistency might disappear, though this should be 
contrasted with the actual choices made by drivers. 
In any case, Allais' paradox and other empirical violations of the postulates of expected utility 
theory have motivated alternative theories of choice. We consider two of them in the next subsections.

\subsection{Dual theory of choice}\label{dualchoice}
While expected utility introduces risk-aversion by magnifying the effects of
bad outcomes through a nonlinear transformation of the cost $c(X)$, Yaari's 
dual theory of choice \cite{yaari1} uses the idea that
a risk-averse agent tends to overstate the probability of bad 
outcomes. % and understates the good ones. 
An agent is then characterized by a continuous nondecreasing 
{\em distortion map} $h:[0,1]\to[0,1]$ with $h(0)=0$ and $h(1)=1$, so that the probability $\PP(X>x)$ is 
distorted as $h(\PP(X>x))$. Risk-aversion corresponds to $h(x)\geq x$ for all 
$x\in[0,1]$, while a risk-prone agent satisfies the reverse inequality.

The function $x\mapsto h(\PP(X>x))$ is a decumulative distribution so we may 
find a random variable $X^h$ such that $\PP(X^h>x)=h(\PP(X>x))$, and we 
may describe the agent's preferences $\trieq$ by the functional
\be\label{drm1}
\rho^h(X)=\EE(X^h)
\ee
or more explicitely in terms of the distribution of $X$ 
\be\label{drm2}
\rho^h(X)=\int^0_{-\infty}\!\![h(\PP(X\!>\!x))\!-\!1]dx+\int^{\infty}_0\!\!h(\PP(X\!>\!x))dx.
\ee
This is a law invariant risk measure which is also
positively homogeneous and normalized on constants. It is called a {\em distortion risk measure}.
In particular it is always translation invariant as opposed to the expected utility 
maps $\rho_c$.

A characterization of the preferences that can be represented in this form
is given in \cite{yaari1}. Namely, assuming that all prospects satisfy $X(\omega)\in[0,1]$ 
almost surely, a preorder $\trieq$ on $\X$ can be characterized by a distortion risk 
measure $\rho^h$ if and only if it is law invariant 
{\bf (A$_1$)} and satisfies\vspace{-2ex}
\begin{itemize}
  \item[{\bf (A$^*_2$)}] {\sc $L^1$-continuity:} the sets $\left\{Y\in\X: Y\trieq X\right\}$ and $\left\{Y\in\X: X\trieq Y\right\}$ are closed for
  convergence in the $L^1$-norm.\vspace{-1ex}
  \item[{\bf (A$^*_3$)}] {\sc dual independence:} if $X,Y,Z\in\X$ are pairwise comonotone and $X\trieq Y$ then
  $\alpha X+(1\!-\!\alpha)Z\trieq \alpha Y+(1\!-\!\alpha)Z$ for all $\alpha\in[0,1]$. \vspace{-1ex}
\item[{\bf (A$^*_4$)}] {\sc monotonicity under first-order stochastic dominance:} if $F_X(t)\geq F_Y(t)$ 
for all $t\in\RR$ then $X\trieq Y$.\vspace{-2ex}
\end{itemize}
The main difference with expected utility theory is the substitution of independence by 
dual independence which uses the concept of {\em comonotonicity}. For our purposes it suffices 
to say that $X$ and $Y$ are comonotone iff there is a third variable $U$ and non-decreasing maps $f$ and 
$g$ such that $X=f(U)$ and $Y=g(U)$. An alternative set of axioms which ensure a representation
by distortion risk measures is given in \cite{rusz5} by considering a prospect space $\X$ of bounded 
random variables on a standard atomless probability space and continuity for  the $L^\infty$ norm.

Although the maps $\rho^h$ are always translation invariant, they may fail to be additively
consistent. As we show next the latter is a stringent condition which is only satisfied for
$h(x)=x$ in which case $\rho^h(X)=\EE(X)$. For $h$ concave and twice differentiable 
this result was established in \cite[Luan]{luan1} (see also \cite[Heilpern]{heil1} and \cite[Goovaerts {\em et al.}]{goovaerts1}).
Our proof does not require any {\em a priori} regularity on $h$ beyond continuity and 
monotonicity. It exploits the following elementary fact.
\begin{lemma}\label{lemma: calc}
Let $h:[0,1]\to[0,1]$ be continuous and suppose that for some $0<p<1$
the limit $L=\lim_{q\to0^+}(h(q)-h(pq))/q$ exists. Then $h$ is right differentiable at $0$ with
$h'_+(0)=L/(1\!-\!p)$.
\end{lemma}
\begin{proof}[Proof]
Take $\varepsilon>0$ and choose $\delta>0$ so that for all $q\in(0,\delta)$ we have
$$L-\varepsilon\leq\frac{h(q)-h(pq)}{q}\leq L+\varepsilon.$$
For each $x\in(0,\delta)$ we may take $q=p^jx$ in order to get 
$$(L-\varepsilon)p^j\leq\frac{h(p^jx)-h(p^{j+1}x)}{x}\leq (L+\varepsilon)p^j.$$
Summing over all $j\geq0$ we get a telescopic series that simplifies to
$$\frac{L-\varepsilon}{1-p}\leq\frac{h(x)-h(0)}{x}\leq \frac{L+\varepsilon}{1-p}.$$
from which the conclusion follows since $\varepsilon$ was arbitrary.
\end{proof}
\vspace{1ex}

\begin{theorem}\label{theorem: yaari2}
The only distortion risk measure $\rho^h$ that is additive consistent is $\rho^h(X)=\EE(X)$ 
which corresponds to $h(x)=x$.
\end{theorem}
\begin{proof}[Proof]
Take $B_p$ and $B_q$ independent Bernoullis with success probabilities $p,q\in[0,1]$.
From Lemma \ref{Lema1} we know that $\rho^h$ is additive so that
\be\label{sumita}
\rho_h(B_p+B_q)=\rho_h(B_p)+\rho_h(B_q).
\ee
Denoting $\bar p=1\!-\!p$ and $\bar q=1\!-\!q$ we have
\be\label{pcasos}
\PP(B_p+B_q>x)=\left\{\begin{array}{cl}
1 & \mbox{if } x\leq 0\\
1-\bar p\bar q& \mbox{if } 0\leq x<1\\
pq & \mbox{if } 1\leq x< 2\\[0.5ex]
0 & \mbox{if }  x\geq 2
\end{array}
\right.
\ee
and then using (\ref{drm2}) we find $\rho_h(B_p+B_q)=h(1\!-\!\bar{p}\bar{q})+h(pq)$.
Similarly we get $\rho_h(B_p)=h({}p)$ and $\rho(B_q)=h(q)$ so that (\ref{sumita}) becomes
\be\label{efin} h(1-\bar{p}\bar{q})+h(pq)=h(p)+h(q)\ee
which can also be written as
\be\label{func2}
h(p+q(1\!-\!p))-h(p{})=h(q)-h(pq).
\ee
Since $h$ is monotone we can find $\tilde p\in(0,1)$ such that $h'(\tilde p)$ exists,
so that (\ref{func2}) implies
$$\lim_{q\to 0^+}[h(q)-h(\tilde pq)]/q=h'(\tilde p)(1\!-\!\tilde p)$$
and then Lemma \ref{lemma: calc} gives $h'_+(0)=h'(\tilde p)$.
Using this fact and dividing (\ref{func2}) by $q(1\!-\!p)$ with $q\to 0^+$, it then follows that
$h$ has a right derivative at each point $p\in[0,1)$ and in fact $h'_+(p{})=h'_+(0)$
is constant. 
%Since $h$ is nondecreasing a known result by Lebesgue shows that $h$ 
%can be recovered by integrating $h'_+(\cdot)$. 
It follows that $h$ is Lipschitz continuous and then absolutely continuous so that 
it can be recovered by integrating its derivative.
Hence $h$ is affine and since
$h(0)=0$ and $h(1)=1$ we conclude that $h(\cdot)$ must be the identity map.
\end{proof}

\subsection{Rank-dependent expected utilities}\label{rankutilities}
Expected utility theory and the dual theory of choice are complementary and can be 
combined by considering preference functionals of the form
$$\rho_c^h(X)=c^{-1}(\EE(c(X^h)))=c^{-1}(\EE(c(X)^h))$$
where $c$ is a utility function and $h$ is a distortion map. 
More explicitly
\be\label{drm3}
\rho_c^h(X)=c^{-1}\!\left(\mbox{$\int^0_{-\infty}[h(\PP(c(X)\!>\!x))\!-\!1]dx+\int^{\infty}_0\!h(\PP(c(X)\!>\!x))dx$}\right)\!.
\ee
Note that $\rho_c^h$ does not change under affine transformations of $c(\cdot)$, so we may assume $c(0)=0$.
For $h(x)=x$ we recover
expected utilities, while $c(x)=x$ gives the distortion risk measures.
The functionals $\rho_c^h$ are called {\em rank-dependent expected utilities} and have been 
considered by several authors including Quiggin \cite{quiggin1}, Wakker \cite{wakker1}, and
 Chateauneuf \cite{chate1}, who provide axiomatic characterizations of the 
preorders $\trieq$ that can be represented in this form.  
Note that $\rho_c^h$ is normalized on constants but, just as for expected utilities, they need not 
be translation invariant nor additive consistent. The next result characterizes
when these properties hold. This result was proved in \cite[Luan]{luan1} assuming
$c(\cdot)$ and $h(\cdot)$ twice differentiable and increasing, with $h$ concave and $c$ convex.
An alternative proof was given in \cite[Goovaerts {\em et al.}]{goovaerts1} under the same hypothesis
but assuming in addition that $c(\cdot)$ admits a McLaurin expansion. Our proof rests on
the techniques developed in the previous sections and avoids such {\em a priori} regularity
which is however obtained as a consequence.
\begin{theorem}\label{theorem: wakker} 
A rank dependent expected utility $\rho_c^h$ is translation invariant iff $c(\cdot)$ is an exponential function
or the identity.  Moreover, the only $\rho_c^h$ which are additive consistent are the entropic 
risk measures: $h$ is the identity and $c$ is either an exponential function or the identity.
\end{theorem}
\proof
Let us first assume that $\rho_c^h$ is translation invariant. 
Take $X=zB_p$ with $z>0$ and $B_p$ a Bernoulli. We then have
$$\PP(c(X+m)>x)=\left\{\begin{array}{cl}
1&\mbox{if }x<c(m)\\
p&\mbox{if } c(m)\leq x< c(m+z)\\
0&\mbox{if }x\geq c(m+z)
\end{array}\right.$$
so that using (\ref{drm3}) and distiguishing cases according to the signs of $m$ and $m+z$, we get 
in all situations
$$\rho_c^h(X+m)=c^{-1}\big(h(p{})c(m\!+\!z)+(1\!-\!h(p{}))c(m)\big).$$
In particular for $m=0$ we have $\rho_c^h(X)=c^{-1}(h(p{})c(z))$
so that the translation invariance $\rho_c^h(X+m)=\rho_c^h(X)+m$ yields 
$$h(p{})c(m+z)+(1\!-\!h(p{}))c(m)=c\left(c^{-1}(h(p{})c(z))+m\right).$$
Letting $\varepsilon= c^{-1}\left(h(p{})c(z)\right)>0$ we get the analog of (\ref{eq:cost2})
\be\label{eq:cost2a}
c(\varepsilon)[c(m+z)-c(m)]=c(z)[c(m+\varepsilon)-c(m)].
\ee
In the case $z<0$, noting that $B_p\sim 1-B_{\bar p}$ with $\bar p=1\!-\!p$, we 
may write $X+m=-zB_{\bar p}+(m+z)$ so we get a similar formula for $\rho_c^h(X+m)$
by replacing $p$ by $\bar p$, $z$ by $-z$, and $m$ by $m+z$, namely
$$\rho_c^h(X+m)=c^{-1}\big(h(\bar p{})c(m)+(1\!-\!h(\bar p{}))c(m+z)\big)$$
from which we obtain again (\ref{eq:cost2a}) this time with $\varepsilon=c^{-1}((1-h(\bar p))c(z))<0$.
Proceeding as in the proof of Theorem \ref{theorem: vnm2} we deduce that
$c(\cdot)$ is either an exponential function or the identity map, which proves
our first claim.

Let us assume next that $\rho_c^h$ satisfies the stronger condition of additive 
consistency, and let us show that in this case $h(x)=x$. The case when $c(\cdot)$ is
the identity was settled in the previous section so we just consider the
exponential case $c(x)=[e^{\beta x}-1]/\beta$. Let us consider two independent 
Bernoullis $B_p$ and $B_q$ so that for all $z>0$ we have
\begin{align}\label{eq: exp1}
	\rho_c^h(zB_p+zB_q)=\rho_c^h(zB_p)+\rho_c^h(zB_q). 
\end{align}
The formulas given in the first part of the proof show that the right-hand side is
equal to $c^{-1}\left(h(p)c(z)\right)+c^{-1}\left(h(q)c(z)\right)$. To compute the 
expression on the left we observe that $\PP(c(zB_p+zB_q)>x)=\PP(B_p+B_q>c^{-1}(x)/z)$
and we may use (\ref{pcasos}) to obtain
\begin{eqnarray*}
\rho_c^h(zB_p+zB_q)
&=&\mbox{$c^{-1}\left(\int^{c(z)}_0h\left(1\!-\!\bar{p}\bar{q}\right)dx+\int^{c(2z)}_{c(z)}h\left(pq\right)dx\right)$}\\
&=&c^{-1}\left(c(z)h\left(1\!-\!\bar{p}\bar{q}\right)+[c(2z)\!-\!c(z)]h\left(pq\right)\right).
\end{eqnarray*}
Plugging these formulas into (\ref{eq: exp1}) we have
$$c(z)h\left(1\!-\!\bar{p}\bar{q}\right)+(c(2z)\!-\!c(z))h\left(pq\right)=c\left(c^{-1}(h(p{})c(z))\!+\!c^{-1}(h(q)c(z))\right).$$
Using the exponential form of $c(\cdot)$ the left-hand side is given by
\begin{align*}
\mbox{$\frac{e^{\beta z}-1}{\beta}h\left(1-\bar{p}\bar{q}\right)+\frac{e^{2\beta z}-e^{\beta z}}{\beta}h\left(pq\right)$}
\end{align*}
whereas after some manipulation the right-hand side is seen to be
\begin{align*}
\mbox{$\frac{e^{\beta z}-1}{\beta}\left(h(p)+h(q)+h(p)h(q)[e^{\beta z}-1]\right)$}
\end{align*}
so that the equation simplifies to
$$h(1-\bar p\bar q)+h(pq)-h(p{})-h(q)=e^{\beta z}\left(h(p)h(q)-h(pq)\right).$$
Since this holds for all $z>0$ we deduce that for all $p,q \in[0,1]$ 
\beq
h(p)h(q)&=&h(pq), \\
h(1-\bar{p}\bar{q})+h(pq)&=&h(p{})+h(q).
\eeq
The latter is the same as (\ref{efin}) so we may use the argument in Theorem \ref{theorem: yaari2} to conclude
that $h$ is the identity. 
\endproof

\section{A remark on dynamic risk measures}\label{dynrisk}

Time consistency is a central issue in multistage decision problems 
under risk where decisions are taken sequentially along periods $t=0,1,\ldots,T$. 
It has been thoroughly investigated using the concept
of {\em dynamic risk measures}: a sequence of risk measures, one for each period, usually 
obtained by iterated composition of conditional risk maps that progressively incorporate 
the random information revealed along time (see for instance \cite{cheridito1,detlefsen,pflug,rusz2} and 
references therein). This approach structurally avoids the 
inconsistencies and allows to deal with non-additive risk measures such as VaR or AVaR. Moreover,
the framework provides a dynamic programming recursion that allows
to characterize and eventually compute optimal solutions.

Route choice can also be seen as a sequential decision process where at each step 
the driver is located at an intermediate node where he must chose the next arc to follow. 
This view is actually the basis of most shortest path algorithms. It is then tempting to use dynamic 
risk measures to model route choice under risk. However, the notion of period is not obvious here. 
One option is to take the set of nodes in the network as state space and associate 
periods with link choice decisions so that time corresponds to the number of link choices 
that have been made so far. Naturally, one has to deal with the fact that the 
same node can be reached after different number of steps, depending on the number 
of links in the actual path followed. In particular it is unclear 
how to define the planning horizon $T$, maybe as the maximum number of links in all 
paths connecting the origin to the destination. 

Although one can find ways to frame route choice as a multistage decision 
process, we do not pursue this goal here. Instead, we point out yet another difficulty in this 
approach which has to do with the network representation. We illustrate this with a very 
simple example on the network in figure \ref{fig:fig4} with independent normally distributed
times $X\sim N(10,1)$, $Y\sim N(10,1)$ and $Z\sim N(20,3)$. Consider the non 
additive risk measure $\rho=\mbox{\rm AVaR}_p$, with 
$p$ chosen so that for all normal variables we have $\rho(X)=\EE(X)+ \sigma(X)$.
\begin{figure}[!htb]
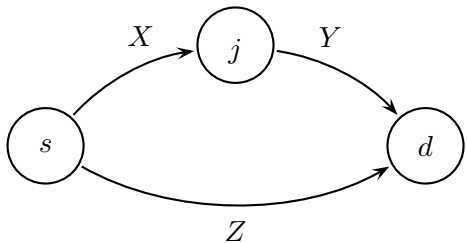

\centering
\psframebox*{%
\rule[-1.0cm]{0pt}{1.0cm}
$\psmatrix[mnode=Circle,radius=5mm,colsep=1.5cm,rowsep=0.3cm,arrowscale=1.5]
& [name=2] j \\
[name=1] s &  &  [name=3] d
\ncarc[nodesep=1pt, arcangle=20]{->}{1}{2}^{X}
\ncarc[nodesep=1pt, arcangle=20]{->}{2}{3}^{Y}
\ncarc[nodesep=1pt, arcangle=-30]{->}{1}{3}_{Z}
\endpsmatrix$ }
\caption{Counterexample with normal distributions.}
\label{fig:fig4}
\end{figure}
By independence, an iterated dynamic risk measure will 
evaluate the risk for the upper route as 
$\rho(X+\rho(Y|X))=\rho(X+\rho(Y))=\rho(X)+\rho(Y)=22$. This is larger than 
$\rho(Z)=20+\sqrt{3}$ which is then the optimal choice.
Suppose now that we merge both upper links into a single arc with 
time $U=X+Y\sim N(20,2)$, which is just a matter of how we decide 
to model the network. In this case the upper route has risk
$\rho(U)=20+\sqrt{2}$ and has displaced $Z$ as the optimal solution.
The conflict arises since there is no clear notion of period to guide our
choice of the representation of the network. This will occur whenever one 
deals with non-additive risk measures, while for additive risk measures the
conflict disappears.

\section{Related work}\label{relatedwork}

Risk-sensitive route choice is a relatively new research area which has been 
growing steadily in the last decade or so. General discussions on risk evaluation 
in the context of route choice can be found in Bates {\em et al.} \cite{bates}, Noland \cite{noland}, 
and Hollander \cite{hollander}. A mean-stdev risk model for atomic and non-atomic traffic equilibrium 
was investigated by Nikolova and Stier-Moses \cite{stier1}, distinguishing the 
case when only the expected values depend on the traffic intensity from the more 
difficult case where also the variance is flow-dependent.
A similar traffic equilibrium model was considered by Ordo\~nez and Stier-Moses in \cite{ordonez1}, 
in which risk-aversion is treated by aggregating a variability index to the expected value. This is compared 
to a model based on $\alpha$-percentiles as well as a novel approach that uses 
ideas from robust optimization. Since computing an $\alpha$-percentile equilibrium is difficult, 
they investigate two classes of approximations which provide a better fit than a standard 
Wardrop model. Percentile equilibria in route choice were also investigated by Nie \cite{nie3}.

Algorithms to compute optimal routes for the mean-stdev objective were studied 
by Nikolova \cite{nikolova1} and Nikolova {\em et al.} \cite{nikolova2}. Despite the combinatorial nature of the problem and 
the nonlinear objective function, an exact algorithm with sub-exponential 
complexity $n^{O(\log n)}$ is found. The main difficulty here comes from the non-additivity of the 
standard deviation. As illustrated by the example in the Introduction, the optimality of a 
path is not inherited by its subpaths,  which prevents the use of dynamic programming 
and makes the problem much more difficult to solve. In contrast, when using additive 
consistent risk measures this difficulty disappears and route choice reduces to a standard 
shortest path problem.

A different approach to risk-averse path choice considers user preferences based 
on the on-time arrival probability. This was studied by Nie and Wu \cite{nie1}, addressing 
the question of whether or not route optimality is inherited by subpaths.  An algorithm for this 
objective function was also given by Nikolova {\em et al.} \cite{nikolova2}. A related approach by Nie {\em et al. }
\cite{nie2,wu1} reconsiders the route choice question under stochastic dominance constraints.

Our contribution partially differs from the previous ones as it uses risk measures to 
quantify route preferences. Exploiting the axiomatic frameworks provided by 
the theories of choice we showed that, in wide classes of risk functionals, the 
entropic risk measures emerge as the only ones that guarantee a form 
of consistency in route choice. In this light, all the models discussed above 
are susceptible to exhibit inconsistencies. While this raises a serious question 
about the capacity of these models to capture rational behavior, it 
does not invalidate them. From a practical viewpoint, all these models
may plausibly describe the behavior of some drivers and ---after all---  no one has yet 
proved that drivers are actually consistent in their decisions! 
From a theoretical perspective our results require the preferences to be defined 
and satisfy additive consistency throughout the space $L^\infty(\Omega,\F,\PP)$. 
This might be asking too much as one could argue that drivers are only able to make 
choices in a much narrower subset of random variables which might not be even a 
linear subspace ({\em e.g.} a set of uniformly bounded non-negative variables).
In summary, while our contribution reveals some strong and interesting consequences 
of additive consistency, there is still work to be done before one 
can provide firm recommendations as to which is the most appropriate way to
model route choice under risk.

\bibliographystyle{acm}
\bibliography{bibliography}
\end{document}